\documentclass[11pt,reqno]{amsart}

\usepackage{amssymb}
\usepackage{amsmath}
\usepackage{amsthm}
\usepackage{latexsym}
\usepackage{amsfonts}
\usepackage{mathrsfs}
\usepackage{bbm}
\usepackage{titlesec}
\usepackage{color}
\usepackage{hyperref}
\usepackage{url}
\usepackage{xcolor} 

\newtheoremstyle{neu_thm}
{13pt}       
{8pt}      
{\itshape}  
{}          
{\bfseries} 
{.}         
{.5em}      
{}          

\newtheoremstyle{neu_defn}
{13pt}       
{8pt}      
{}  
{}          
{\bfseries} 
{.}         
{.5em}      
{}          

\newtheorem{theorem}{Theorem}[section]

\newtheorem{lemma}[theorem]{Lemma}
\newtheorem{proposition}[theorem]{Proposition}

\newtheorem{conjecture}[theorem]{Conjecture}
\theoremstyle{neu_defn}

\titleformat{\section}{\normalfont\bfseries\centering}{\thesection.}{.25em}{}
\titleformat{\subsection}{\normalfont\bfseries}{\thesubsection.}{.25em}{}
\titlespacing{\section}{0pt}{*5}{*1.5}
\titlespacing{\subsection}{0pt}{*4}{*0.5}
\numberwithin{equation}{section}
\allowdisplaybreaks

\newcommand{\R}{\ensuremath{\mathbb R}}    
\newcommand{\Z}{\ensuremath{\mathbb Z}}    

\begin{document}
\title[On the principal minors of Fourier matrices]{On the principal minors of Fourier matrices}

\author[A. Caragea]{Andrei Caragea$^1$}
\address{$^1$KU Eichst\"att-Ingolstadt, MIDS, Goldknopfgasse 7, 85049 Ingolstadt, Germany}
\email{andrei.caragea@gmail.com}
\urladdr{https://sites.google.com/view/andrei-caragea}

\author[D. G. Lee]{Dae Gwan Lee$^2$}
\address{$^2$Department of Applied Mathematics, Pukyong National University, Yongso-ro 45, Nam-Gu, 48513 Busan, Republic of Korea}
\email{daegwan@pknu.ac.kr}

\begin{abstract}
For the $N$-dimensional Fourier matrix $\mathcal F_N$, we prove that if $N \geq 4$ is square-free, then every $2\times 2$ and $3\times 3$ principal minor of $\mathcal F_N$ is nonzero. We also show that if $N \geq 4$ is not square-free, then $\mathcal F_N$ has zero principal minors of all sizes. Moreover, based on numerical experiments, we conjecture that if $N$ is square-free, then all principal minors of $\mathcal F_N$ are nonzero.
\end{abstract}



\maketitle
\thispagestyle{empty}

\section[\hspace*{1cm}Introduction and main results]{Introduction and main results}

The Fourier matrix, which expresses the discrete Fourier transform (DFT) as a transformation matrix, is widely used in many areas of mathematics and engineering, including discrete signal processing. 
Along with its unitarity up to a scaling factor, which allows the DFT to be inverted without losing signal information, the Fourier matrix possesses many interesting and useful properties. 
For instance, the restricted isometry property of random Fourier submatrices has been extensively utilized in compressed sensing \cite{CRT06,FR13}, the conditioning of cyclically contiguous submatrices of the Fourier matrix (or, more generally, Vandermonde matrices) has been studied in the context of super-resolution \cite{AB19,Ba22,Ba00,LL21}, and maximal rank-deficient submatrices of a Fourier matrix with prime power dimension have been explored in connection with uncertainty principles \cite{DV08}.

One notable property is that all minors of a Fourier matrix with prime dimension are nonzero, a result proven by Chebotar\"ev in the 1920s \cite{SL96}. 
In \cite{KN15}, Kozma and Nitzan highlight the connection between invertibility properties of Fourier submatrices and the construction of Riesz exponential bases. More recently, with the introduction of so-called hierarchical Riesz bases of exponentials in \cite{PRW24,CL22}, the importance of understanding whether principal minors of the Fourier matrix (with potentially permuted columns or rows) are nonzero has been emphasized in \cite[Conjecture~2]{LPW23} and \cite[p.\,36, Conjecture~2]{C24}.

In this paper, we investigate the invertibility of $2\times 2$ and $3\times 3$ principal submatrices of the Fourier matrix. 
Using elementary arguments, we establish a connection between the square-freeness of $N$ (i.e., the property that $N$ is divisible by no square number other than $1$) 
and the existence of zero $2\times 2$ and $3\times 3$ principal minors of the $N$-dimensional Fourier matrix. 
In particular, we show that if $N$ is not square-free, then zero principal minors of any size between $2$ and $N-2$ do in fact exist.
We further conjecture that, if $N$ is square-free, all principal submatrices of the $N$-dimensional Fourier matrix are invertible.

Before presenting our results, we introduce some necessary notations and terminology. 
The \emph{$N$-dimensional Fourier matrix} (or the \emph{$N$-point DFT matrix}) is given by $\mathcal F_N:=\big(\omega^{k\ell}\big)_{0\leq k,\ell\leq N-1}$, where $\omega=e^{2\pi i/N}$ is a principal $N$-th root of unity. 
The \emph{principal submatrix} of $\mathcal F_N$ associated with an index subset $K\subset\{0,1,\dots,N-1\}$ is given by $\mathcal F_N[K]:=\big(\omega^{k\ell}\big)_{k,\ell\in K}$.
Varying the index subset $K$ recovers all principal submatrices of $\mathcal F_N$.
A \emph{principal minor} refers to the determinant of a principal submatrix. 
With a slight abuse of notation, we will simply call the determinant of an $r\times r$ principal submatrix an $r\times r$ principal minor.

Let us first discuss the case of low dimensions.  
For the Fourier matrix $\mathcal F_N$ with $N = 1,2,3$, it is easy to check that every principal minor is nonzero. 
However, in dimension $N =4$, the Fourier matrix 
\[
\mathcal F_4
= \begin{pmatrix} 1 & 1 & 1 & 1 \\ 1 & i & -1 & -i \\ 1 & -1 & 1 & -1 \\ 1 & -i & -1 & i \end{pmatrix}
\]
has $2 \times 2$ singular principal submatrices 
\[
\mathcal F_4 [\{0,2\}]
= \begin{pmatrix} 1 & 1 \\ 1 & 1 \end{pmatrix} \quad \text{and} \quad 
\mathcal F_4 [\{1,3\}] 
= \begin{pmatrix} i & -i \\ -i & i \end{pmatrix},
\]
while all of its other $2\times 2$ and all $1 \times 1$ and $3 \times 3$ principal submatrices are nonsingular (Proposition \ref{prop:complementarity} below sheds more light on the singularity of complementary principal minors).

Our main results generalize these observations to higher dimensions. 

\begin{theorem}\label{thm:main1}
Let $N\geq 4$ be a square-free number. Then for each $r\in\{2,3,N-3,N-2\}$, every $r\times r$ principal minor of $\mathcal F_N$ is nonzero. 
\end{theorem}

\begin{theorem}\label{thm:main2}
Let $N\geq 4$ be not square-free, i.e., $N$ is divisible by a square number other than $1$. Then, for each $2\leq r\leq N-2$, there exists a zero $r\times r$ principal minor of $\mathcal F_N$.
\end{theorem}  

While working on the proof the main result of \cite{CL22}, we realized that the existence of hierarchical Riesz bases of exponentials for unions of intervals with rational endpoints would follow directly from invertibility properties of certain minors of Fourier matrices. This lead to the formulation of Conjecture $2$ in \cite{LPW23}. Based on subsequent numerical and symbolic experiments in Python, we arrived at the following conjecture in the previous arXiv version of this manuscript. 

\begin{conjecture}[Caragea, Lee, Malikiosis, Pfander]\label{conj:main}
\,
\begin{itemize}
   \item[a.] For any natural number $N$, there exists a permutation $\sigma$ on $\{0,1,\dots,N-1\}$ such that the Fourier matrix with permuted columns $\big(\omega^{k\sigma(\ell)}\big)_{0\leq k,\ell\leq N-1}$ has no zero principal minors.
   \item[b.] If $N$ is square-free, then the identity permutation suffices, that is, the Fourier matrix $\mathcal F_N$ has no zero principal minors (also stated earlier in \cite{CMN24}). 
   \item[c.] If $N$ is not square-free, then the Fourier matrix $\mathcal F_N$ has a zero $r \times r$ principal minor for all $2\leq r\leq N-2$.   \end{itemize}
\end{conjecture}

Theorem \ref{thm:main2} solves part $c.$ of Conjecture \ref{conj:main} in the affirmative. After correcting a numerical error in our Python code and substantial refinement of the algorithm used to search through the set of possible permutations, we were also able to disprove part $a.$ of Conjecture \ref{conj:main}. We found that for $N=16$ there are no `good' permutations and that this is the smallest $N$ for which this occurs. A detailed discussion of the numerics involved and the algebraic insights gained working on this will be presented in a different manuscript, currently in preparation.

We would like to note that after posting this paper on arXiv, we received a correspondence drawing our attention to the now published paper \cite{CMN24} which is closely related to our work and was uploaded to arXiv in April 2024. 
Therein, Cabrelli, Molter and Negreira also arrived at part $b.$ of Conjecture \ref{conj:main} (and, to our knowledge, this is the first explicit formulation of this conjecture in the literature), and provided a similar proof for Lemma \ref{lem:indexreduction} below. They also argued, in virtually the same way as we do in the proof of Theorem \ref{thm:main1}, that if $N$ is square-free, then all $2\times 2$ principal minors of $\mathcal F_N$ are nonzero. We decided to still include the proofs of these statements in the current manuscript for the reader's convenience. It is interesting to note that while our motivation to investigate principal submatrices of the Fourier matrix comes from the study of hierarchical exponential Riesz bases, Cabrelli, Molter and Negreira arrive upon this problem from a wish to generalize the idea of weaving bases of finite dimensional vector spaces to the generality of weaving different stable generator sets for shift-invariant spaces. 

Preliminary investigations show that, already for the case of $4 \times 4$ principal minors, elementary proof techniques are likely no longer sufficient to resolve part $b.$ of Conjecture~\ref{conj:main}. In fact, the existence of a zero $4\times 4$ principal minor leads to an equation very similar to \eqref{eq:33eq} below, namely that 

\[\Big(r_A e^{i \pi  \frac{A}{N}}-1\Big)\Big(r_B e^{i \pi \frac{B}{N}}-1\Big)=\Big(r_C e^{i \pi \frac{C}{N}} -1\Big)^2\]
for some integers $A,B,C$. However, the fact that the radii $r_A, r_B, r_C$ need not be $1$ significantly complicates any attempt at a contradiction derived from purely trigonometric considerations. We expect that
deeper number theoretic and cyclotomic polynomial tools are needed to settle part $b.$ of Conjecture~\ref{conj:main}.

\section{Auxiliary results}

In this section we note a result involving complementary principal minors and an elementary lemma which essentially states that, when looking for zero principal minors, it suffices to assume that the first index is $0$. 

\begin{proposition}[Proposition 3 in \cite{L25}]\label{prop:complementarity}
Let $K\subset\{0,1,\dots,N-1\}$. Then the principal submatrix $\mathcal F_N[K]$ is singular if and only if the complementary principal submatrix $\mathcal F_N[K^C]$ is singular, where $K^C = \{0,1,\dots,N-1\}\setminus K$.
\end{proposition}

Although Proposition \ref{prop:complementarity} is perhaps folklore and can be proved using frame theory 
(see e.g., Proposition~5.4 in \cite{BCMS19} or Proposition~2.1 in \cite{MM09}), 
we were made aware of an elementary proof based on the Jacobi complementarity theorem in a discussion with Loukaki (see \cite{L25}). 

\begin{lemma}[Proposition 2.13 in \cite{CMN24}]\label{lem:indexreduction} 
Let $K=\{k_1,k_2,\dots,k_r\}\subset\{0,1,\dots,N-1\}$ and set $L=\{0,k_2-k_1,\dots,k_r-k_1\}$ where the differences are understood modulo $N$. Then $\mathcal F_N[K]$ is singular if and only if $\mathcal F_N[L]$ is singular.
\end{lemma}

\begin{proof}
Direct computation using elementary row and columns operations shows that 
\[
\begin{split}
&\begin{vmatrix} \omega^{a_1^2} & \omega^{a_1a_2} & \dots & \omega^{a_1a_r} \\ \omega^{a_2a_1} & \omega^{a_2^2} & \dots & \omega^{a_2a_r} \\ \vdots & \vdots & \ddots & \vdots \\ \omega^{a_ra_1} & \omega^{a_ra_2} & \dots & \omega^{a_r^2}
\end{vmatrix} \\
=\; & \omega^{a_1(-ra_1+2\sum_{j=1}^ra_j)} 
 \cdot \begin{vmatrix} 1 & 1 & \dots & 1 \\ 1 & \omega^{(a_2-a_1)^2} & \dots & \omega^{(a_2-a_1)(a_r-a_1)} \\ \vdots & \vdots & \ddots & \vdots \\ 1 & \omega^{(a_r-a_1)(a_2-a_1)} & \dots & \omega^{(a_r-a_1)^2}\end{vmatrix}.
\end{split}
\] 
\end{proof}



\section{Proof of Theorem \ref{thm:main1}}

In view of Proposition \ref{prop:complementarity}, we only need to deal with $r=2$ and $r=3$. Furthermore, due to Lemma \ref{lem:indexreduction}, we only need to consider the matrices $\mathcal F_N[\{0,a\}]$ for $0<a \leq N-1$ to settle the $r=2$ case, and the matrices $\mathcal F_N[\{0,a,b\}]$ for $0<a<b\leq N-1$ to settle the $r=3$ case. 

Note that since $N$ is square-free, the equation $n^2\equiv 0\mod{N}$ cannot hold for $0<n\leq N-1$; otherwise, $N$ would divide $n^2$, and the square-freeness of $N$ would imply that $N$ divides $n$.

If $0=\det \mathcal F_N[\{0,a\}]=\begin{vmatrix} 1 & 1 \\ 1 & \omega^{a^2}\end{vmatrix} = \omega^{a^2} -1$ with $0<a \leq N-1$, then $a^2 \in N\Z$, which leads to a contradiction as discussed above.  
Therefore, we have $\det \mathcal F_N[\{0,a\}] \neq 0$ for all $0<a \leq N-1$ (Corollary 2.14 in \cite{CMN24}).

In the case of $3\times 3$ principal minors, direct computation shows that
\[
\begin{split}\det\mathcal F_N[\{0,a,b\}] = \begin{vmatrix} 1 & 1 & 1 \\ 1 & \omega^{a^2} & \omega^{ab} \\ 1 & \omega^{ab} & \omega^{b^2}\end{vmatrix} & = \omega^{a^2+b^2} +2\omega^{ab} -\omega^{a^2}-\omega^{b^2}-\omega^{2ab} \\
& = \big(\omega^{a^2}-1\big)\big(\omega^{b^2}-1\big) - \big(\omega^{ab}-1\big)^2,
\end{split}\]
so the condition $\det\mathcal F_N[\{0,a,b\}]=0$ is equivalent to 
\begin{equation}\label{eq:33eq}
\big(\omega^{a^2}-1\big) \big(\omega^{b^2}-1\big) = \big(\omega^{ab}-1\big)^2.
\end{equation}
Note that for any $\theta\in[0,2\pi)$, we have 
\[\begin{split}
e^{i\theta}-1 &= e^{i\frac{\theta}{2}}\left(e^{i\frac{\theta}{2}}-e^{-i\frac{\theta}{2}}\right) \\
& =e^{i\frac{\theta}{2}}\Big(\cos\tfrac{\theta}{2}+i\sin\tfrac{\theta}{2}-\cos\big(-\tfrac{\theta}{2}\big)-i\sin\big(-\tfrac{\theta}{2}\big)\Big) \\
& = 2i\sin\tfrac{\theta}{2}e^{i\frac{\theta}{2}}.
\end{split}\]
Using this identity, \eqref{eq:33eq} becomes 
\[
\sin\big(\pi\tfrac{a^2}{N}\big)e^{\pi i\frac{a^2}{N}}\cdot \sin\big(\pi\tfrac{b^2}{N}\big)e^{\pi i\frac{b^2}{N}}=\sin^2\big(\pi\tfrac{ab}{N}\big)e^{\pi i\frac{2ab}{N}},
\]
that is, 
\[
\sin\big(\pi\tfrac{a^2}{N}\big) \sin\big(\pi\tfrac{b^2}{N}\big) e^{\pi i\frac{(b-a)^2}{N}} = \sin^2\big(\pi\tfrac{ab}{N}\big) .
\]
If $\sin^2\big(\pi\tfrac{ab}{N}\big) = 0$, then we must have $\sin\big(\pi\tfrac{a^2}{N}\big)=0$ or $\sin\big(\pi\tfrac{b^2}{N}\big)=0$, which is equivalent to $a^2 \in N\Z$ or $b^2 \in N\Z$. 
Otherwise, if $\sin^2\big(\pi\tfrac{ab}{N}\big) \neq 0$, we must have $e^{\pi i\frac{(b-a)^2}{N}} = \pm 1\in\R$, which is equivalent to $(b-a)^2 \in N\Z$. 
Since $a, b, b-a$ are in $\{ 0, 1, \ldots, N-1 \}$, any one of the conditions $a^2 \in N\Z$, $b^2 \in N\Z$, $(b-a)^2 \in N\Z$ leads to a contradiction as discussed above.  
Therefore, we conclude that 
$\mathcal F_N[\{0,a,b\}] \neq 0$ for all $0<a<b\leq N-1$. 
\hfill $\Box$ 

\section{Proof of Theorem \ref{thm:main2}}

We may write $N=p^2m$ for some prime $p \geq 2$ and some number $m \geq 1$. 

First, assume that $p=2$, i.e., $N = 4m$.
According to Proposition~\ref{prop:complementarity}, it is enough to prove the statement for $r$ satisfying $2\leq r\leq \frac{N}{2} = 2m$. 
Choose any set $L \subset \{ 0, 2, 4, \ldots, 2(2m-1) \}$ of cardinality $r$ containing $0$ and $2m$, and consider the $r \times r$ principal submatrix $\mathcal F_N[L]$ of $\mathcal F_N$. 
Since $\omega^{4m} = 1$, the rows of $\mathcal F_N[L]$ that are indexed by $0$ and $2m$ are both the vector containing only ones, i.e.,  
\[R_{0}=\Big(\omega^{0\cdot\ell}\Big)_{\ell\in L} = \big( 1 \big)_{\ell\in L} 
\quad \text{and} \quad 
R_{2m}=\Big(\omega^{2m\cdot\ell}\Big)_{\ell\in L} = \big( 1 \big)_{\ell\in L} . 
\]
Therefore, $\mathcal F_N[L]$ is singular.

Now, assume that $p \geq 3$. Again, due to Proposition~\ref{prop:complementarity}, we only need to prove the statement for $r$ satisfying $2\leq r\leq \frac{N}{2}$. 
%
%
%
%

\medskip

\noindent
\textbf{Case $2\leq r\leq p$.} \ 
Consider the set $Q=\{0,pm,2pm,\dots,(p-1)pm\}$ which is of cardinality $p$. Since $(kpm)\cdot(\ell pm)\equiv 0\mod p^2m$ for all $k, \ell$, the entries of $\mathcal F_N[Q]$ are all equal to $1$, and therefore, every principal submatrix of $\mathcal F_N[Q]$ is singular. This implies that the matrix $\mathcal F_N$ has a zero $r\times r$ principal minor for every $2\leq r\leq p$. 

\medskip

\noindent
\textbf{Case $p<r\leq\frac{N}{2}$.} \ 
For $j = 0 , 1, \ldots, p-1$, let 
\[K_j:= \{ kp+j : k=0,1, \ldots, pm-1 \} , \]
which is a set of cardinality $pm$.
Since $p<r\leq\frac{N}{2}$, there exist unique integers $0\leq s \leq \frac{p}{2}$ and $0\leq t<pm$ such that $r=spm+t$. 
Define
\[L:= K_0\cup K_1\cup \cdots \cup K_{s-1}\cup K_{s,t}, 
\] 
where $K_{s,t} := \{ kp+s : k=0,1, \ldots, t-1 \}$ is the subset of $K_s$ containing the first $t$ elements of $K_s$. 
Clearly, the set $L$ is of cardinality $r$. Since $r>p$, we can always ensure that $\{0,pm,2pm,\dots,(p-1)pm\}\subset L$.

We claim that the $r \times r$ matrix $\mathcal F_N[L]$ is singular. 
To see this, let us consider the rows of $\mathcal F_N[L]$ that are indexed by $\{0,pm,\ldots,(p-1)pm\} \subset K_0$, namely the rows 
\[R_{k pm}=\Big(\omega^{k pm\cdot\ell}\Big)_{\ell\in L}, 
\quad k = 0, 1, \ldots , p-1 . 
\]
By setting $\xi = \omega^{pm}=e^{2\pi i/p}$, we have 
\[R_{kpm} = \Big(\underbrace{1,1,\dots,1}_{pm},\underbrace{\xi^k,\xi^k,\dots,\xi^k}_{pm},\dots,\underbrace{\xi^{(s-1)k},\xi^{(s-1)k},\dots,\xi^{(s-1)k}}_{pm},\underbrace{\xi^{sk},\xi^{sk},\dots,\xi^{sk}}_{t}\Big). \]
Regardless of whether $t$ is zero or nonzero, these row vectors must be linearly dependent if the vectors 
\[v_k:=\Big(1,\xi^k,\xi^{2k},\dots,\xi^{(s-1)k},\xi^{sk}\Big),\quad k = 0, 1, \ldots , p-1,\]
are linearly dependent. 
Note that since $s\leq\frac{p}{2}$ and $p \geq 3$, we have $s+1 \leq\frac{p}{2} +1 < p$. 
As the vectors $v_k$, $k = 0, 1, \ldots , p-1$, live in dimension $s+1$, they are necessarily linearly dependent; thus, the matrix $\mathcal F_N[L]$ is singular.  
This completes the proof. 
\hfill $\Box$ 


\vspace{.2cm}
\noindent {\bf Acknowledgements.} A.~Caragea was supported by the German Research Foundation (DFG) Grant PF 450/11-1 and Grant CA 3683/1-1.
D.G.~Lee was supported by the National Research Foundation of Korea (NRF) grant funded by the Korean government (MSIT) (RS-2023-00275360).

The authors would like to thank Romanos D.~Malikiosis and G\"otz E.~Pfander for helpful discussions and simplifications towards the exposition of this document, Maria Loukaki for the elegant proof of Proposition \ref{prop:complementarity}, and Carlos Cabrelli for discussions concerning \cite{CMN24}. The authors would also like to thank Florian Lange and Aditi Bhushan for valuable support with numerical experiments.

Both authors contributed equally to this work and are jointly first authors. 

\vspace{-.5cm}
\section*{Author affiliations}
\end{document}